\documentclass[letterpaper,12pt]{article}
\usepackage{graphicx}
\usepackage{amsthm, amsmath, amssymb}
\usepackage{mathrsfs,latexsym}
\usepackage{verbatim}
\newtheorem{thm}{Theorem}[section]
\newtheorem{cor}[thm]{Corollary}
\newtheorem{lemma}[thm]{Lemma}
\theoremstyle{remark}
\theoremstyle{definition}
\newtheorem{defn}[thm]{Definition}

\numberwithin{equation}{section}
\theoremstyle{definition}
\newtheorem{exmp}[thm]{Example}

\begin{document}
%
\title{Levels of Ultrafilters with Extension Divisibilities}
\author{Salahddeen Khalifa\\Department of Mathematics and Computer Sience\\University of Missouri- St. Louis\\St.Louis, MO 63121 U.S.A\\e-mail: skkf57@mail.umsl.edu}
%
%
\maketitle
	\begin{abstract} 
		To work more accurately with elements of the semigroup of the Stone Cech compactification $ (\beta N , .) $ of  the discrete semigroup of natural numbers N under multiplication . We divided these elements into ultrafilters which are on finite levels and ultrafilters which are not on finite levels. For the ultrafilters that are on finite levels we prove that any element is irreducible or product of irreducible elements and all elements on higher levels are $ {\tilde \mid } $ -divided by some elements on lower levels. We characterize ultrafilters that are not on finite levels and the effect of $ {\tilde \mid } $ -divisibility on the ultrafilters which are not on finite levels. 
	\end{abstract} 
\section{Introduction} 
If we given the discrete semi-group (N,.) of the discrete space of natural numbers with multiplication , then $ (\beta N, .)  $ is a semigroup of the Stone-Cech compactification $ \beta N $ of the discrete space N with operation(.) that is defined as: 
For any $ x , y \in \beta N , A \in x . y $ iff $ { \{n \in N : \,   A / n \in y} \} \in x $ where \newline $  A / n = {\{m \in N : m n \in A\} }.$ \newline  The topology on $ \beta N $ has  base that contains the basic (clopen) sets that are \newline defined as : for any $ A \subseteq N$ ,$\overline A = {\{ x \in \beta N : A \in x\} }. $ For each $ n \in N $ the principal ultrafilter is defined by element ${\{n}\}$. $ N^* = {\beta N} - N $ and $ A^* = {\overline A} - A $ for $ A \subseteq N $. $\mid\beta N\mid =2^c$ where $ c=\mid\mathbb{R}\mid$ .\newline  The collection of \textbf{upward closed} subsets of N is $ \mu = {\{A \subseteq N : A = A \uparrow \}} $ where $ A \uparrow = {\{ n \in N : \exists a \in A , a \mid n \} } $ and the collection of \textbf{downward closed} subsets of N is $ \nu = {\{ A \subseteq N : A = A \downarrow \}} $ where $ A \downarrow = {\{ n \in N : \exists a \in A , n \mid a \}} .\newline $ For every $ f : N \rightarrow N $ there is a unique continuous extension function \newline $ \tilde f : \beta N \rightarrow \beta N $ such that for every $ x \in \beta N , \tilde f (x) = {\{ A \subseteq N : f ^{-1} (A)  \in x }\} $. If $ A\in x $, then $ f(A)\in \tilde f(x) $ and if $ B\in\tilde f(x) $, then $ f^{-1}(B)\in x $.\newline For $ x\in\beta N $,if $ A\in x $, then $x\upharpoonright A ={\{ B \cap A : B \in x \}}  $ is an ultrafilter on A. Let $ x,y\in\beta N$. Then: $ y $ is \textbf{left-divisible} by $ x$, $ x\mid_l y $ if there is $ z\in\beta N$ such that $ y=zx $. $y$ is \textbf{right-divisible} by $ x, x\mid_r y$, if there is $ z\in\beta N$ such that $y=xz$. $ y $ is \textbf{mid-divisible} by $ x, x\mid_m y$, if there are $ z,w\in\beta N$ such that $ y=zxw$. $ y $ is \textbf {$ \tilde\mid$-divisible} by $ x $, $ x \tilde \mid y $ if for all $ A \subseteq N , A \in x $  implies $ \mid A = {\{n \in N : \exists a \in A , a \mid n \} } \in y $. When $x=n \in N$, we write $ n\mid y, (y=nz, z\in\beta N) $. 
\begin{lemma} 
([1]Theorem 4.15)Assume $ x, y \in \beta N $ and $ A \subseteq N. $ Then $ A \in x y $ if and only if there exist $ B \in x $ and an indexed family $ <C_n>_{n\in B} $ in $ y $ such that $ \bigcup_ {n\in B} n C_n \subseteq  A. $ 
\end{lemma}
\begin{lemma}
	([3]Lemma 1.4)The following conditions are equavalent. \newline 
	(a) $ x \tilde \mid y $ \newline 
	(b) $ x \, \cap\mu \subseteq y \cap \mu $ \newline 
	(c) $ y \, \cap \nu \subseteq x \cap \nu $   
		\begin{lemma} 
		([4]Lemma 2.1)Let $ x \in \beta N $, $ A\in x $ , and $ f : N \rightarrow N $ \newline  
		(a) If $ f(a) \mid a $ for all $ a \in A $ , then $ \tilde f (x) \tilde \mid x $ \newline 
		(b) If $ a \mid f (a) $ for all $ a \in A , $ then $ x \tilde \mid \tilde f (x) $ \newline \newline
		In Lemma 1-3, if $ A \in x $ , in order to determine $ \tilde f (x) $ it is enough to know values of $ f (a) $ for $ a \in A $ , i.e we will sometimes define the function only on a set in x.  
	\end{lemma} 
		  An element ${ p\in \beta N -\{1}\} $ is prime with respect to $ \mid_l,\mid_r,\mid_m$ or $\tilde \mid  $ divisibility if it is divisible only by 1 and itself. We call p $ \in \beta N $  irreducible in $ X \subseteq \beta N $ if it  can not be represented in the form $ p = x y $ for $ x , y \in X -\{1\}.$  
		\begin{lemma}
			(a) $ p\in\beta N$ is prime with respect to $ \mid_l,\mid_r$ and $ \mid_m$ divisibilities if and only if $p$ is irreducible.\newline
			(b) ([4]Lemma 2.2) $ p \in \beta N $ is prime with respect to $\tilde\mid$-divisibility if and only if $ p \in \overline P $ (P the set of prime numbers in N). \newline 
			(c) ([2]Lemma7.3) If $ p \in \beta N $ and $ p \in\overline P $ , then p is irreducible in $\beta N $. \newline (d) If $ p\in\beta N$ is prime with respect to $\tilde\mid$- divisibility. Then $p$ is prime with respect to $\mid_l,\mid_r$ and $\mid_m$ divisibilities.
		\begin{proof}
		(a) This follows immediately from definition of prime and irreducible.  \newline
		(d)By (a),(b),(c)
		\end{proof}
		\end{lemma} 
	\end{lemma}
\section {Ultrafilters on finite levels in $ \beta N $}  
\theoremstyle{definition}
In this section we consider the set of ultrafilters which are elements in all basic open sets $\overline L_i ,i=0,1,2,......$ where $\ L_0=\{1\} $ and $\ L_n=\{a_1a_2......a_n :a_1,a_2,......,a_n\in P\}$, P the set of prime numbers in N.
\begin{defn}
	
([4] definition 2.4)Let P be the set of prime numbers in N, and 
let \newline  \, \, \, \, $ L_0 = \{1\} $ \newline 
$ L_1 = \{a : a \in P \} $ \newline
$ L_2 = \{a_1 a_2 : a_1 , a_2 \in P \} $ 
\\
. 
\\ 
.
\\
$ L_n = \{a_1 a_2 ......a_n : a_1 , a_2, ....,a_n \in P\} $ \newline . \\ .\\
Then, the ultrafilter $ x$  is called on "finite level" if it is in exactly one of the follwing sets: $ \overline L_i  , i = 0 , 1, 2 ,....... $ where $ \overline L_i = \{y \in \beta N : L_i \in y \} $ 
\begin{defn}  
	Let $ A \subseteq P $. We denote $ A ^ n = \{a ^ n : a \in A \}$  and \newline $ A ^ {(n)} = \{ a_1 a_2 ..... a_n  :$ all $ a_i \in A \,  are \,  mutually \, prime \,  numbers \} $  

\begin{lemma} 
(1) $ \overline L_i \cap \overline L_j = \emptyset $  for any $ \, i \neq j \newline $ 
(2) $ \overline L_i = \overline {\rm P^i} \cup \overline {\rm P ^ {i - 1} P} \cup ..... \cup \overline {\rm P ^ {(i)}} $ for any $ i \geq 2 $ , and  $\overline {\rm P ^i}  ,  \overline {\rm P ^ {i-1} P}  , ......,   \overline {\rm P ^ {(i)}} $ are disjoint\newline 
(3) $ \overline L_i $ has $2 ^c$ elements for all $ i \neq 0$ \newline 
(4) All $ \overline {\rm P ^i}  , \overline {\rm P ^ {i - 1}P} $ , ....., $ \overline  { \rm P ^ {(i)}} $  has $ 2 ^ c $ elements \newline  
(5) Any principal ultrafilter in $ \beta N $ is ultrafilter on finite level. i.e $ N\subset \bigcup\limits_{i=0}^{\infty} \overline L_i $\newline 
(6) If $ A \subseteq N $ is finite, then all ultrafilters $ x \in \overline A $ are on finite level. \newline 
 
 \end{lemma} 
\begin{proof} 
(1) Assume $ \overline L_i \cap \overline L_j \neq \emptyset $ for some $ i \neq j $ , so there exist $ x \in \beta N $ such that $ x \in \overline L_i \cap \overline L_j $ , so $ L_i , , L_j \in x $ and $ L_i \cap L_j = \emptyset\in x $ a contradiction.Thus $ \overline L_i \cap \overline L_j\ =\emptyset $ for any $ i \neq j $\newline 
(2) To avoid cumbersome notation we prove this in case of $ i = 3 $. Ultrafilters in the 3rd level contains $ L_3 = {\{ a_1 a_2 a_3 : a_1, a_2, a_3 \in P \} = \{8, 12, 18, 20, 27,.....\}} $ and we can partition $L_3$ as: 
\begin{align*}
L_3 &= {\{8, 27 , 125 , ......\}} \cup \{12, 18, 20, 28......\} \cup \{30, 42, 66......\}\\ 
&= P ^ 3 \cup P^2 \ P \cup P ^{(3)} 
\end{align*}
So $ \overline L_3 = \overline { \rm P^3 \cup P^2 \, P \cup P ^ {(3)}} = \overline { \rm P ^3 } \cup \overline {\rm P ^ 2} P \cup \overline {\rm P {(3)}}  $.  Since $ P ^3 , P ^2 P , P ^ {(3)} $ are disjoint , then $ \overline {\rm P ^ 3} , \overline {{\rm P ^ 2} P} $ and $ \overline {\rm P ^{(3)}} $ are disjointed \newline
(3) Since $ L_i $ is infinite set and $ \overline L_i = L_i \cup L_i ^ * $ so by([5]Theorem 3.3) $ \overline L_i $ has $ 2 ^ c$ ultrafilters \newline 
(4) If we prove this in case of $ i = 3 $. All $ P ^ 3 , P ^ 2 \, P , P ^ {(3)} $ are infinite sets. Thus $ \overline {\rm P ^ 3} , \overline {\rm P ^ 2 \, P} $ and $ \overline {\rm P ^ {(3)}} $ have $ 2 ^ c $ elelments \newline 
(5) Since for any $ n \in N $ is prime number or product of prime numbers. So there exist $L_n$ such that $ n \in L_n$ and since $ \overline L_n = L_n \cup L_n ^ *, $ so $ n \in \overline L_n $ as principal ultrafilter.Thus $ N\subset\bigcup\limits_{i=0}^{\infty} \overline L_i $ \newline 
(6) Let $ A = \{n _1 , n _2 , ........, n _m\} $ be a finite subset of N. For any $ x \in \overline A $, we have ${\{n_1.n_2,.....,n_m}\}\in x $, so ${x=n_i}$ for exactly of one of $ 1 \leqslant i \leqslant \, m $.Thus $x$ is principal ultrafilter, and by (5) $x$ is on finite level.\newline  
\end{proof} 
\begin{defn}	
(1)([4]Definition 5.1) We call ultrafilters of the form $ p ^ k $ for some $ p \in \overline P $ and $ k \in N $ \textbf{basic}. Let $\mathcal {B}$ be the set of all basic ultrafilters and let $\mathcal {A}  $ be the set of all functions  $ \alpha : \mathcal{B} \rightarrow N \cup \{0\} $ with finite support $(\ \{b \in \mathcal {B} : \alpha\, {(b)} \neq 0  \} $  is finite)
i.e $ \alpha = \{(b_1 , n_1) (b_2, n_2)........ (b_m, n_m)\} $ , $ \alpha (b) = 0 $ for $ \alpha \notin \{b_1 , b_2 ,...., b_m\} $ . Let $ \alpha = \{ ({p_1} ^ {k_1}, n_1) , ({p_2} ^{k_2} , n_2), ....., ({p_m}^ {k_m} , n_m) \} \in \mathcal {A}  $ , $ (p_i \in \overline P) $ \newline Set $ F_\alpha =  \{({{A_1} ^ {k_1}}) ^ {(n_1)}  ({{A_2} ^ {k_2}}) ^ {(n_2)} ....... ({{A_m} ^{k_m}}) ^ {(n_m)} : A_i \in p_i \upharpoonright P , A_i = A_j \, if \newline\ p_i = p_j \, \, A_i \cap A_j = \emptyset \, \, otherwise \} $ \newline where $ ({{A_1} ^ {k_1}}) ^{(n_1)} ({{A_2} ^{K_2}}) ^ {(n_2)} ...... ({{A_m} ^ {k_m}}) ^{(n_m)} = \{ $$\prod\limits _ {i = 1} ^ {m} \prod\limits _ {j=1} ^ {n_i} a_{i,j} ^ {k_i}  : a_{i,j} \in A_i \newline for \ all \ i, j \ and \ a_{i,j} are \, \, distinct\}  $$ $ \newline\newline 
(2)([4]Definition 3.1) If $pow _n :  N \rightarrow N $ is defined by $ pow_n(a)=\ a^n $ then for $ x \in N^*$ , $ \widetilde {pow _n }(x) $ is generated by sets $ A^n $ for $A \in x$ . we will denote $ \widetilde {pow _n} (x) \ $, with $ x ^n $  \newline  
\begin{exmp}
If $ \alpha = \{(p^2, 1)\} , $ then $ F_\alpha = \{A^2 : A \in p  \upharpoonright P \} $.\newline If $ \alpha = \{ (p, 2) \} $ then $  F_\alpha = \{A^{(2)} : A \in p \upharpoonright P \} $. If $ \alpha = \{ (p^3,2) , (q^2,2) \}  $ , then $ F_\alpha = \{ {(A^3)} ^{(2)} {(B^2)}^{(2)}:A\in p \upharpoonright P , B \in q \upharpoonright P, A \cap B = \emptyset\}   $ 
\end{exmp}
\begin{defn} 
	([4] Definition 5.4)If $ \alpha=\{ ({p_1}^ {k_1}, n_1 ),({p_2}^{k_2},n_2),..,({p_m}^{k_m},n_m)  \}\in \mathcal{A} , p_i \in \overline P ,$ we denote $ \sigma (\alpha) = \displaystyle \sum_ {i=1} ^ {m} k_i n_i  $ 
\end{defn}

\end{defn} 
\end{defn} 
\end{defn}
\begin{thm}
	 
 $ F_\alpha \subseteq x $ for some $ \alpha \in \mathcal {A} $ if and only if $x$ is n-th level $ \overline L_n $ (for $ n \in N$) such that $ \sigma (\alpha) = n $ 
\end{thm}
\begin{proof}
$ (\Rightarrow) $ Let $ \alpha = \{ ({ p_1} ^ {k_1} , n_1) , ({ p_2} ^ {k_2}, n_2),......, ({p_m} ^ {k_m}, n_m) \} \in \mathcal {A} $\newline with $ \sigma (\alpha) = n $ and $ F_\alpha \subseteq x $ , so $ ({{A_1}} ^{k_1})^{(n_1)}  ({{A_2}} ^{k_2}) ^ {(n_2)}...... ({{A_m}} ^ {(k_m)}) ^ {(n_m)} \in x , $  and since $ ({{A_1}} ^ {k_1}) ^ {(n_1)} ({{A_2}} ^ {k_2}) ^{(n_2)} ({{A_m}} ^{k_m}) ^ {(n_m)}  \subseteq L_n. $ So $ L_n \in x $ and $ x \in \overline L_n $  \newline 
$ (\Leftarrow)$ By ([4] Theorem 5.5) 
\end{proof} 
\begin{thm}
(a) If $ x \in \overline L_m $ and $ y \in \overline L_n .$ Then $ x \, y \in \overline {L_{m + n}}$ \newline 
(b) If $ F_\alpha \subseteq x $ and $ F_\beta \subseteq y . $ Then $ F_{\alpha +\beta}  \subseteq x y $ where \newline  $ \alpha + \beta = \{(b, n+ n^ {\prime}) : (b, n) \in \alpha  , (b, n^ {\prime} ) \in \beta \} $ 
\end{thm}
\begin{proof}
(a) Let $ x \in \overline L_m $ and $ y \in \overline L_n , $ , then $L_m\in x, L_n\in y$. In (Lemma 1-1) if we put $ B = {L_m} $ and for all $n\in L_m$ ,$C_n=L_n$,then we have \newline ${L_mL_n}={\bigcup\limits_{n\in L_m} nL_n}\in xy$, $xy\in\overline{L_mL_n}$. Thus $xy\in\overline {L_{m+n}}$ \newline
(b) Let $ F_\alpha \subseteq x $,$F_\beta\subseteq y $ and $ \alpha , \beta \in \mathcal {A} , $ so $ x $ is the m-th level $ \overline L_m $ and $y$ is the n-th level $\overline L_n $ , where $\sigma (\alpha) = m $ and $ \sigma (\beta) = n , $ so by (Theorem 2.7) and (a) $ xy  \in \overline {L_{m+n}} $ where $ m+n = \sigma (\alpha + \beta),$ and also by ( Theorem 2-7 ) we have $ F_ {\alpha + \beta} \subseteq x y. $ 
\end{proof}  
\begin{lemma} 
$ : \bigcup\limits_{i=0} ^ {n} \overline L_i = \overline{\bigcup\limits_ {i=0} ^{n}} L_i  $
\end{lemma}
\begin{proof} 
Let $x \in \beta N , x \in \bigcup\limits_ {i=0} ^ {n} \overline L_i $ then $ x \in \overline L_i $ for some $i $ , so $ L_i \in x$ , $ \bigcup\limits_ {i=0} ^ {n} L_i \in x , $ and $ x \in  \overline {\bigcup\limits_ {i=0} ^ {n}} L_i  $ . Thus $ \bigcup\limits_ {i=0} ^ {n} \overline L_i \subset \overline {\bigcup\limits_ {i=0} ^{n} }$. If $ x \in \overline {\bigcup\limits_ {i=0} ^{n} } L_i $, then $\bigcup\limits_{i=0}^{n}L_i\in x $, so  $L_i\in x $ for some  $ i $ and $ x\in\overline L_i $. Thus $ x \in \bigcup\limits_ {i=0} ^ {n} \overline L_i $ , so $ \overline {\bigcup\limits_ {i=0} ^ {n}} L_i  \subset \bigcup\limits_ {i=0} ^ {n} \overline L_i $. Therefore \newline $ \bigcup\limits_ {i=0} ^{n} \overline L_i = \overline {\bigcup\limits_ {i=0} ^ {n} L_i }$  
\end{proof}
Since the irreducible elements are the prime elements with respect to $\mid_l,\mid_r$ and $\mid_m$-divisibilities, then the notion that any natural number is either a prime number or a product of prime numbers transfer from N to the set of ultrafilters which are on finite levels.
\begin{thm}
	Any ultrafilter $ x \in\beta N$ on finite level $ (x\in\overline L_i) $ where $ i \geq 1 $ is irreducible or product of irreducible elements.
\end{thm}
\begin{proof}
	(1) If $ x\in\overline L_1 $ , then by (Lemma 1-4 (c)) $x$  is irreducible.\newline\
	(2) If $x\in\overline L_2 $: For $ x \in L_2 $, $ x $ is product of two prime numbers, so by (Lemma 1-4 (c)) $ x $ is product of irreducible elements. Let $ x\in {L_2}^*$,$ x=y z \in {L_2}^*, \newline (L_2 \in  x= yz )$, then $ {\{m \in N : L_2/m \in z}\} \in y $, and since $ L_2/m = \{r \in N : rm \in L_2 \} $. So we have four  cases for values of $L_2/m$ such as :
	\begin{align*} 
	L_2 /m = \{r \in N : rm \in L_2 \} &= 1  \text{ when } m  \text{ is in }L_2 \rightarrow (1) \\
	&= L_2  \text { when } m =1 \rightarrow (2)\\   
	&= L_1 \text { when } m \text { is in }L_1 \rightarrow (3)\\    
	& = \emptyset \text { when } m \text { is }  \text { otherwise }  \rightarrow (4)  
	\end{align*} 
	Now : By (1) and (2) we have either $ y= 1 $ or $z= 1$ , so $x$ is irreducible. And by (3) $ x$ is product of two prime elements so by (Lemma 1.4 (c) ) $x$ is product of irreducible elements. \newline \newline
	(3) If $ x \in \overline L_3 $: For $ x\in L_3$, $ x $ is product of three prime numbers, so by (Lemma 1-4 (c)) $ x$ is product of irreducible elements. Let $ x\in {L_3}^*$, $ x = y z \in {L_3}^* (L_3 \in x = y z) $, then ${\{m \in N : L_3/m \in z}\}\in y $. and sine $L_3 /m = \{ r \in N : rm \in L_3 \} $. So we have five cases for values of $L_3/m$.such as:
	\begin{align*} 
	L_3 /m = \{ r \in N : rm \in L_3 \}& = 1 \text { when } m \text { is in }  L_3 \rightarrow (1)\\   
	&= L_3 \text { when }  m=1 \rightarrow (2)\\
	&= L_2 \text { when } m \text { is in } L_1 \rightarrow (3)\\
	&= L_1 \text { when } m \text { is in }  L_2 \rightarrow (4)\\
	&=\emptyset \text { when } m \text { is otherwise }  \rightarrow (5) 
	\end{align*} 
	Now : By (1) and (2) we have either $ y= 1 $ or $ z= 1$ . Thus , $x$ is irreducible. By (3) $ L_1 \in y , L_2 \in z, $ we have two cases : (a) $ y$ is irreducible and $z$ is irreducible. Therefore , $x$ is product of irreducible elements . (b) $y$ is irreducible , and $z$ is product of irreducible elements. Therefore , $x$ is product of irreducible elements . Similar to (3) , we have by (4) $x$ is product of irreducible elements. \newline \newline 
	(4) If we continue in this way and we suppose that any element $ x \in \overline L_i $ , $\ 1 \leqslant i \leqslant n - 1 $ is irreducible or product of irreducible elements. Then we can prove that for any $ x \in \overline L_n $ is irreducible or product of irreducible elements such as : \newline \newline If $ x \in \overline L_n $: For $ x\in L_n $, $ x $ is product of n times prime numbers, so by ( Lemma 1-4 (c)) $ x $ is product of irreducible.Let $ x\in {L_n}^*$, $ x = y z \in {L_n}^* (L_n \in x = y z) $, then ${\{m \in N : L_n / m \in z}\}\in y$, and sine $L_n /m = \{ r \in N : rm\in L_n \}$. So we have n+2 cases for values of $L_2 / n$ such as: 
	\begin{align*}
	L_ n /m = \{ r \in  N : m r\in L_n \} &= 1 \text { when } m  \text { is in }  L_n \rightarrow (1)\\
	& = L_n \text { when } m=1 \rightarrow (2) \\
	& = L_{n-1} \text { when } m \text { in } L_1\rightarrow (3)\\
	& = L_{n-2} \text { when } m \text { is in }  L_2 \rightarrow (4)\\
	& . \\
	& .\\
	& .\\
	& .\\  
	& = L_2 \text { when } m \text { is in }  L_{n-2} \rightarrow (n)\\ 
	& = L_1 \text { when } m \text { is in } L_{n-1} \rightarrow (n+1)\\ 
	& = \emptyset \text { when } m \text { is otherwise }  \rightarrow (n+2) 
	\end{align*}
	Now  By (1) and (2) we have either $ y=1 $ or $z=1$ , so $x$ is irreducible . Moreover,in all the other cases we have $x$ is product of irreducible elements. 
\end{proof}
\begin{cor}
(a) For any n-th level, $ \overline L_n $ has $ 2^c$ irreducible elements.
\newline (b) For any n-th level, $ \overline L_n$ there are $ 2^c$ ultrafilters in $ L^*$ that are not irreducible ( product of irreducible elements ).
\end{cor}
\begin{proof}
(a) Since for any $ L_n\subset N$ is infinite,so by ( [2] proposition 7.4 ) there exist infinite set $ A\subset L_n $, such that all elements of $ A^* $ are irreducible. Also since by ([ 5] Theorem 3-3 ) $ A^*$ has $ 2^c$elements and $ A^*\subset L_n^*$. Thus $ \overline L_n$ has $ 2^c$ irreducible elements. \newline (b) Let $x$ and $y$ are distinct elements of $ \overline L_1 $ and let $ p\in L_{n-1}^*$.Since for every disjoint subset $A$ and $B$ of $ L_1$ we have $ \overline A p\cap\overline B p =\emptyset $, then by ( [1] Theorem 8.11 (5),(3)), $  xp \neq\ yp $ [$xp , yp\in L_n^*$ by (Theorem 2-8 (a)  )]. Thus $ L_n^*$ has $ 2^c$ ultrafilters which are product of irreducible elements.
\end{proof}
\begin{cor}
If $ x\in\overline L_n, n\geq 2 $ is not irreducible ultrafilter, then there exist at least two ultrafilters $ x_i\in\overline L_i, x_j\in\overline L_j,  i,j<\ n, x=x_ix_j $.
\end{cor}
\begin{proof}
Let $ x\in\overline L_n$, and $ x $ is not irreducible, then by ( Theorem 2.10 ) $x$ is product of at least two irreducible elements $ x_i,x_j, i,j<\ n, x_i\in\overline L_i, \newline x_j\in\overline  L_j,  L_iL_j=L_n$ and  $x=x_ix_j\in\overline L_i\overline L_j\subset\overline {L_iL_j}=\overline L_n$.
\end{proof}
\begin{cor}
For any $ x\in\overline L_i, i\geq 2 $ which is not irreducible, there exist at least two ultrafilters $ x_i\in\overline L_i$ and $ x_j\in\overline L_j, i,j<\ n$,such that: $ x_i\mid_lx$, $ x_j\mid_rx$,and $ x_i\mid_mx$, $ x_j\mid_mx$.

\end{cor}
The following theorem shows that the facts that for any $m\in L_m$ there is $n\in L_n$ where $n\leq m$ such that $n\mid m$, and for any $n\in L_n$ there is $m\in L_m$ such that $n\mid m$ can be transfered to  the $\tilde \mid$-divisibility on the ultrafilters that are on finite levels
\begin{thm}
(a) For every ultrafilter $ x \in \beta N \ {-} \overline {L_0 \cup L_1 \cup ...... \cup L_{n-1}} $ on finite level, there is an ultrafilter $ y \in \overline L_n $ such that $ y \tilde \mid x$ \newline 
(b) For any ultrafilter $ x \in \overline L_m $ , there exist an ultrafilter $ y \in \overline L_n, m \leqslant n $ such that $ x \tilde \mid y$  
\end{thm} 
\begin{proof}
(a)Let $ x\in\beta N-\overline{ L_0\cup\ L_1\cup......\cup\ L_{n-1}} $ is ultrfilter on finite level and let $ f:\ N -\ L_0 \cup L_1 \cup......\cup L_{n-1}\longrightarrow N $ is defined by $f(n)$ be the smallest factor of $n$ in $L_n$.So we have $ x\notin\overline {L_0\cup L_1\cup......\cup L_{n-1}}$ and $ L_0 \cup L_1 \cup ....... \cup L_{n-1} \notin x $ and we have $ N- L_0 \cup L_1 \cup ........ \cup L_{n-1} \in x, f(N- L_0 \cup L_1 \cup ....... \cup L_{n-1}) \in \tilde f (x) , $ and since $ f ( N-{{L_0}} \cup L_1 \cup ........ \cup L_{n-1}) \subseteq L_n , $ then we have $ L_n \in \tilde f (x) ,\newline \tilde f (x) \in \overline L_n .$ Since for any $n \in N- L_0 \cup L_1 \cup ....... \cup L_{n-1} $ by definition of the function $f$ we have $ f(n) \mid n. $ Then by (Lemma 1.3 (a)) we have $ \tilde f (x) \tilde \mid x$, $ \tilde f(x)\in\overline L_n $. \newline  
(b) Let $ x\in\beta N $, $ x\in\overline L_m $, and let  $ f : L_m \rightarrow N $ be defined by $ f (n) $ is the smallest multiple of $n$ in $L_n$ , so we have $ L_m \in x , f (L_m) \in \tilde f (x) $ , and since $ f (L_m) \subseteq L_n $ then we have $ L_n \in \tilde f (x) , \tilde f (x) \in \overline L_n $ . Since for any $ n \in L_m $ by definition of the function $f$ we have $ n \mid f(n) .$ Then by (Lemma1.3 (b)) we have $ x \tilde \mid \tilde f (x) $, $ \tilde f(x) \in\overline L_n $ \newline 
\end{proof}
\begin{cor}
(a) For any ultrafilter $ x\in\overline L_n $ there are ultrfilters $x_i\in\overline L_i $, $  i  \leqslant n-1 $ such that $ 1 \tilde\mid x_1......\tilde\mid x_{n-1} \tilde\mid x $ \newline 
(b) For any ultrafilter $ x_m\in\overline L_m $ there exist a sequence $ < x_n : x_n\in\overline L_n , n \geq m > $ such that $x_m\tilde\mid x_n$ 
\end{cor}
\begin{proof}
(a)Let $ x\in\overline L_n $, then by (Theorem 2.14 (a)) there are ultrafilters \newline $x_i\in\overline L_i $ $ i\leqslant n-1 $ such that $ 1 \tilde \mid x_1......\tilde\mid x_{n-1} \tilde \mid x $ \newline
(b)Let  $ x_m \in \overline L_m $ then by (Theorem 2.14 (b))  there exist a sequence \newline $ <x_n : x_n \in \bar L_n , n \geq m > $ such that $ x_m \tilde \mid x_{m+1} \tilde \mid x _{m+2} \tilde \mid ..........$   
\end{proof}

\section{Ultrafiters that are not on finite levels}
In order to look for the ultrafilters $x$ that are not on finite levels($x\notin L_i, i=0,1,.....$),we use the facts that the set of all basic open sets $\mathcal{B}=\{\overline A : A\subseteq N\}$ is a base for the space $\beta N$ and $\{\overline L_i , i=0,1,......\}\subset \mathcal{B}$. From these we can find a basic open set $\overline A\in \mathcal{B}$ such that all nonprincipal ultrafilters $x\in A^*$ will not be elements in any basic open set $\overline L_i , i=0,1,......$  
 \begin{lemma}  
 (a) There are $2^c$ ultrafilters $x$ that are not on finite levels: i.e \newline $ x\notin \bigcup\limits_{i=0}^{\infty}\overline L_i$.\newline(b) There are $2^c$ irreducible ultrafilters $x$ that are not on finite levels.\newline(c)  $ \bigcup\limits_ {i=0} ^ {\infty }\overline L_i \neq\overline {\bigcup\limits_ {i=0} ^  {\infty}} L_i $  
\end{lemma}
\begin{proof}
Let $ A = \{ n_0 , n_1 , n_2 .........\} $ where $ n_i \in L_i $ . Any ultrafilter $x$ that is in finite level in $\overline A$ is principal , because, if $ x \in \overline L_i $ for some $i$ and $ x \in \overline A $ , Then $ L_i \in x $ and $ A \in x $ , so $ L_i \cap A \in x , $ but $ L_i \cap A = \{ n_i\} , \{ n_i\} \in x $ so $ x = \{ n_i\} $ is principal ultrafilter . Now, since $\overline A $ is closed and by ([5]Theorem 3.3 ) any closed subset of $ \beta N $ has finitely many or $ 2 ^c $ elements , so $ \overline A $ has $ 2 ^c $ elements. Thus $ \overline A $ has $ 2 ^ c$ nonprincipal ultrafilters that are not on finite levels.\newline (b)If we take A as in (a), then by ([2]proposition 7.4) there is infinite set B such that $B\subseteq A$ and all elements $x\in B^*$ are irreducible. \newline 
(c) by (a) there are $2 ^ c$ ultrafilters $x$ such that $ x \notin \overline L_i $ for all $i = 0 , 1 , 2 , .......$ so $ \bigcup\limits_ {i=0} ^ {\infty} \overline L_i \neq \beta N =\overline{\bigcup\limits_{i=0}^{\infty}} L_i $ 
\end{proof}
 In (Lemma 3-1 (a)) for any ultrafilter $x\in\beta N $ , $x\in A^*$, we have $L_i\notin x$ for all $i\in N$,so $N-L_i\in x$, $x\in\overline {N-L_i}$ and $x\in\bigcap\limits_ {i=0}^ {\infty} \overline {\rm N-L_i} $
\begin{defn}
$ I = \bigcap\limits_ {i=0} ^ { \infty} \overline { \rm N-L_i}  $ 
\end{defn} 
The ultrafilters that are belong to $I$ in (Definition (3-2)) are called the ultrafilters that are not on finite levels.
\begin{lemma}
(a) An ultrafilter $ x \in I $ if and only if $ x \notin \overline L_i $ for all $ i=0 , 1, 2 .... $ \newline 
(b) $ \beta N - I = \bigcup\limits_ {i=0} ^ { \infty} \overline L_i $  
\end{lemma}
\begin{proof}
(a)  $ (\Rightarrow) $  Let $ x \in \beta N , x \in I , $ so $ x \in \overline {\rm N-L_i } $ for all $ i = 0 , 1 , 2 , ......, $ so $ N- L_i \in x , $ $ L_i \notin x $ . Thus $ x \notin \overline L_i $ for all $ i =0 , 1, 2 , ...... $ \newline $(\Leftarrow)$ Let $ x \notin \bar L_i $ for all $ i=0 , 1 , ......, $ so $ L_i \notin x , N-L_i \in x , $ so $ x \in \overline {\rm N-L_i}$ for all $ i=0, 1 , ..... .$ Thus $ x \in \bigcap \limits_ {i=0} ^ {\infty}\overline {\rm N-L_i} $\newline 
(b)By ([1]Lemma 3.17 (c)) we have
\begin{align*}
  \beta N- I &= \beta N - \bigcap\limits_ {i=0} ^ {\infty} \overline {\rm N- L_i}\\
 & = \bigcup\limits_ {i=0} ^ {\infty} \beta N - (\overline {\rm N-L_i})\\
 \end{align*}
 \begin{align*}  
 & = \bigcup\limits_ {i=0} ^{\infty} \beta N-(\beta N -\overline L_i)\\  
 & = \bigcup\limits_ {i=0} ^ {\infty} \overline L_i  
 \end{align*}
\end{proof} 
\begin{lemma}
If $ x \in I $ . Then \newline (a) $ x \notin \overline {\bigcup\limits_ {i=0} ^ {n}} L_i $ \newline (b) $ x \in \overline {\bigcup\limits_ {i=n} ^ {\infty}} L_i $  
\end{lemma}
\begin{proof}
(a) Let $ x \in I , $ then $ x \notin \bigcup\limits_ {i=0} ^ {n} \overline L_i $ and by (Lemma 2-9) we have \newline $ x \notin \overline {\bigcup\limits_ {i=0} ^ {\ n} } L_i $ \newline (b) Since $ x \in \overline {\bigcup\limits_ {i=0} ^ {\infty} } L_i , $ so $ \bigcup\limits_ {i=0} ^ {\infty}	L_i \in x, L_0 \cup (\bigcup\limits_ {i=0} ^ {\infty} L_i) \in x , $ so $ \bigcup\limits_ {i=0}^ {\infty} L_i \in x. $ Again $ L_1 \cup (\bigcup\limits_ {i=2} ^ {\infty} L_i) \in x , $ and we have $ \bigcup\limits_ {i=2} ^ {\infty} L_i \in x . $ If we continuous in this we will get $\bigcup\limits_ {i=n} ^ {\infty} L_i \in x . $ Thus $ x \in \overline {\bigcup\limits_ {i=n }^ { \infty} }L_i $  
\end{proof}
In particular, any union of infinite  elements $L_i, i=0,1,2,......$ is element in any ultrafilter $x\in I$. This fact leads us to prove that the elements in $I$ are $\tilde \mid$-divisible by elements of any finite level $\overline L_i, i=0,1,......$, as the following theorem shows.
\begin{thm}
 (a) For any ultrafilter $ x \in I $, there exist an ultrafilter $ y\in \overline L_n $ such that $ y \tilde \mid x $ \newline (b) For any ultrafilter $ x \in \ L_n^* $ there exist an ultrafilter $ y \in I $ such that $ x \tilde \mid y $ \newline (c) There exist an ultrafilter $ x \in I $ divided by an ultrafilter $ y \in I $, ($ y\tilde\mid x $). 
\end{thm} 
\begin{proof} 
(a) Let $ x \in \beta N , x \in I , $ so by (Lemma 3-4 (b)) we have 
 $ x \in \overline {\bigcup\limits_ {i=n} ^ {\infty}}   L_i $ and  $ \bigcup\limits_ {i=n} ^ {\infty} L_i \in x .$ Let $ f : \bigcup\limits_ {i=n} ^ { \infty} L_i \rightarrow N $ is defined by $ f(n) $
 is the smallest factor of $n$ in $ L_n $ , so $ f (\bigcup\limits_ {i=n} ^ {\infty} L_i) \subset L_n , $ and since $ f (\bigcup\limits_ {i=1}^ {\infty} L_i ) \in \tilde f (x) ,$ so $ L_n \in \tilde f (x) , \tilde f (x) \in \overline L_n .$ Therefore by (Lemma 1.3 (a)) $ \tilde f (x) \tilde \mid x , \tilde f(x) \in \overline L_n $ \newline \newline (b) Let $ x\in\beta N $, $ x\in \overline L_n $ and let $ A = \{ m_n , m_{n+1} , ........\} \subset N $ where $m_i \in L_i , i \geq n $ such that for any $ n \in L_n $ has multiple in $A$ . Let $ f: L_n \rightarrow N $ be surjective function and it is defined by $f(n)$ is the smallest multiple of $n$ in $A$. If $ x \in\ L_n^* $ , so $ L_n \in x $ ,$ f(L_n ) \in \tilde f (x) $ , and since $ f (L_n) \subset A $ ,  so $ A \in \tilde f (x) , \tilde f (x) \in A ^*$.[ ${\tilde f(x)\in\ I} $, because $x$ is nonprincipal ultrafilter,so any element of $x$ is infinite subset of $N$, since $f$ is injective, then also any subset of $\tilde f(x)$ is infinite subset of $N$.Thus $\tilde f(x) $ is nonprincipal ultrafilter, so by definition of $\tilde f(x) $ and (Lemma 3-1 (a)) $ \tilde f(x)\in\ I $].Therefore by (Lemma 1.3 (b)) $ x\tilde \mid \tilde f(x) $, $ \tilde f(x)\in I $.\newline \newline(c) Let $ x \in\beta N $ and let $ A_1=\{3.2^n :n=1,2,......\} $, $ A_2=\{2^n : n=1,2,......\}$, and $ x \in\ A_1^* $ ,so $A_1\in\ x $, and by (Lemma 3-1 (a)) we have $ x\in\ I $. Let $ f: A_1\rightarrow N $ is defined by $ f(3.2^n) =2^n $, $ n=1,2,......$ $f(A_1)\in\tilde f(x) $ and since $f(A_1)\subset A_2 $, so $A_2\in\tilde f(x) $, $ \tilde f(x)\in\overline A_2 $.Then similar to analogous in (c) we have $\tilde f(x)\in\ I $. Therefore, by (Lemma 1.3 (a) ) we have $\tilde f(x)\tilde\mid\ x $, $\tilde f(x)\in\ I $.
\end{proof}
\begin{cor}
 For any ultrafilter $x\in I $, there exist a sequence \newline $< x_n : n\in N > $ of ultrafilters such that $ x_n\in\overline L_n $ and $ x_1\tilde\mid x_2......\tilde\mid x$
\end{cor}
\begin{proof}
  By (Theorem 3-5 (b)) for any $ x\in I $ there exist an ultrafilter \newline $x_n\in\overline L_n $ for any finite level such that $ x_n\tilde\mid x $,and by (Thmeorem 2-14 (a)) for any $ x_n\in\overline L_n $ there exist an ultrafilter $ x_{n-1}\in\overline L_{n-1} $ such that $ x_{n-1}\tilde\mid\ x_n $.Therefore, there exist a sequence $< x_n : n\in N >$ such that $ x_1\tilde\mid x_2......\tilde\mid x $  
\end{proof}
 
\begin{thm}
(a) If $x, y \in \beta N $ and $ x , y \in I $ . Then $ x y \in I $ and $ y x \in I $ \newline (b) If $ x , y \in \beta N $ and $ x \in I , y \notin I $ . Then $ x y \in I $ and $ y x \in I $ . 
\end{thm}
\begin{proof}
	(a) If we assume that $ x y \notin I$,then  $x y \in \overline L_n $ for some $ n\in N $.So by (Theorem 2.7) we have $ F_\alpha \subseteq x y $ for some $ \alpha \in \mathcal {A} $ such that  $\sigma (\alpha) = n $ where \newline  $ \alpha = \{ ({p_1}^{ k_1} , n_1) , ({p_2} ^ {k_2} , n_2) , ...... ({p_m} ^ {k_m} , n_m) \} , n = \displaystyle \sum_ {i=1} ^ {m} k_i n_i$, \newline $ F_\alpha = \{ ({{A_1}^{k_1}}) ^ {(n_1)} \  ({{A_2} ^ {k_2}})^ {(n_2)} ({{A_m} ^ {k_m}}) ^ {(n_m)} : A_i \in p_i \upharpoonright  P , A_i \cap A_j = \emptyset \ if \\ p_i \neq p_j \} \subseteq x y $ \newline\newline so $ ({{A_1} ^ {k_1}}) ^ {(n_1)}  ({{A_2} ^ {k_2}}) ^ {(n_2)} ...... ({{A_m} ^ {k_m}})^{(n_m)} \in x y ,$ and we have\newline \newline  $ \{ r \in  N : ({{A_1} ^{k_1}})^ {(n_1)}\}  ({{A_2} ^ {k_2}})^ {(n_2)} ....... ({{A_m} ^ { k_m}}) ^ {(n_m)} / r  \in y \} \in x $ , but \newline \newline  $ ({{A_1} ^ { k_1}}) ^ {(n_1)}  ({{A_2} ^{k_2}}) ^ {(n_2)} .... ({{A_m} ^ {k_m}}) ^ {(n_m)}  / r = \{ s \in N : r s \in ({{A_1}^ {k_1}}) ^ {(n_2)}  .... ({{A_m} ^ { k_m}}) ^ {(n_m )}\} $ \newline \newline so $ ({{A_1} ^ {k_1}}) ^ {(n_1)}  ({{A_2} ^ {k_2}}) ^ {(n_2)}...... ({{A_m} ^{k_m}}) ^ {(n_m)} / r \subseteq L_i $, \newline \newline and $ \{ r \in N : ({{A_1} ^ {k_1}}) ^ {(n_2)} ..... ({{A_m} ^ {k_m}}) ^ {( n_m)} / r \in y \} \subseteq L_j $ where $ i , j \leq n $ such that $ L_i . L_j = L_{i+j} = L_n, $. So $ L_i \in y , y \in \overline L_i $ for some $ i \leq n $ and $ L_j \in x , x \in \overline L_j $ for some $ j \leq n $, so we have a contradiction. Thus $ x y \in I. $ Same analogues for prove $ y x \in I $ \newline \newline
	(b) Similar to (a)    
	\end{proof}
\begin{cor}
For any $ n \in N$ there exist $ x\in I $ such that $ n\tilde \mid x$.
\end{cor}
\begin{proof}
Let $ n\in N $ and $y\in I$, then by (Theorem 3.7)we have $ ny \in I$, and  since $n\mid ny$ then $n\tilde\mid ny , ny\in I $.
\end{proof}

\end{document}